\documentclass[a4paper]{amsart}
\usepackage[utf8]{inputenc}
\usepackage[english]{babel}
\usepackage{amsfonts}
\usepackage{amssymb}
\usepackage{amsthm}
\usepackage{mathtools}
\usepackage{bm}
\usepackage[left=3cm, right=3cm]{geometry}

\usepackage{tikz}
\usepackage{tikz-cd}
\tikzset{>=latex}
\usepackage{pgfplots}
\pgfplotsset{compat=1.6}
\usetikzlibrary{intersections,patterns,pgfplots.fillbetween}
\pgfdeclarelayer{bg}
\pgfsetlayers{bg,main}
\makeatletter %
\pgfdeclarepatternformonly[\LineSpace]{my north east lines}{\pgfqpoint{-1pt}{-1pt}}{\pgfqpoint{\LineSpace}{\LineSpace}}{\pgfqpoint{\LineSpace}{\LineSpace}}%
{
	\pgfsetcolor{\tikz@pattern@color} %
	\pgfsetlinewidth{0.1pt}
	\pgfpathmoveto{\pgfqpoint{0pt}{0pt}}
	\pgfpathlineto{\pgfqpoint{\LineSpace + 0.1pt}{\LineSpace + 0.1pt}}
	\pgfusepath{stroke}
}
\makeatother %
\newdimen\LineSpace
\tikzset{
	line space/.code={\LineSpace=#1},
	line space=10pt
}

\usepackage{mathtools}
\mathtoolsset{showonlyrefs}

\usepackage{hyperref}
\usepackage{xcolor}
\hypersetup{
	colorlinks,
	linkcolor={blue!90!black},
	citecolor={green!50!black},
	urlcolor={blue!90!black}
}

\newcommand\nnfootnote[1]{%
	\begin{NoHyper}
		\renewcommand\thefootnote{}\footnote{#1}%
		\addtocounter{footnote}{-1}%
	\end{NoHyper}
}

\usepackage[backend=bibtex,style=numeric,doi=false,url=false,isbn=false,sorting=nyt]{biblatex}
\newcommand{\Z}{\mathbb Z}
\newcommand{\N}{\mathbb N}
\newcommand{\Q}{\mathbb Q}
\newcommand{\R}{\mathbb R}
\newcommand{\C}{\mathbb C}
\renewcommand{\P}{\mathbb P}
\newcommand{\mc}{\mathcal}

\renewcommand{\phi}{\varphi}
\newcommand{\n}{\nabla}
\newcommand{\pd}{\partial}

\renewcommand{\ge}{\geqslant}
\renewcommand{\le}{\leqslant}
\newcommand{\mf}{\mathfrak}

\newcommand{\Ad}{\mathrm{Ad}}

\newcommand{\Sym}{\mathrm{Sym}\,}
\newcommand{\End}{\mathrm{End}\,}
\newcommand{\Hom}{\mathrm{Hom}\,}

\newcommand{\la}{\langle}
\newcommand{\ra}{\rangle}
\newcommand{\til}[1]{\widetilde{#1}}

\renewcommand{\bar}[1]{\mkern 0.5mu\overline{\mkern-0.5mu#1\mkern-0.5mu}\mkern 0.5mu}

\theoremstyle{plain}
\newtheorem{theorem}{Theorem}[section]
\newtheorem*{theorem*}{Theorem}
\newtheorem{conjecture}[theorem]{Conjecture}
\newtheorem{lemma}[theorem]{Lemma}

\newtheorem{corollary}[theorem]{Corollary}
\newtheorem{proposition}[theorem]{Proposition}

\newtheorem{question}[theorem]{Question}
\theoremstyle{definition}
\newtheorem{remark}[theorem]{Remark}

\newtheorem{definition}[theorem]{Definition}
\newtheorem{example}[theorem]{Example}

\newcommand{\curvature}{{\Omega}}
\newcommand{\cric}{\mc{S}}

\begin{document}
	\title[On the Structure of Hermitian Manifolds with Semipositive Griffiths Curvature]{On the Structure of Hermitian Manifolds\\ with Semipositive Griffiths Curvature}
	\date{}
	\author{Yury Ustinovskiy}
	\address{Courant Institute of Mathematical Sicences, New York University, 251 Mercer Street, New York, NY, 10012}
	\email{yu3@nyu.edu}
	
	\begin{abstract}
	In this paper we establish partial structure results on the geometry of compact Hermitian manifolds of semipositive Griffiths curvature. We show that after appropriate arbitrary small deformation of the initial metric, the null spaces of the Chern-Ricci two-form generate a holomorphic, integrable distribution. This distribution induces an isometric, holomorphic, almost free action of a complex Lie group on the universal cover of the manifold. Our proof combines the strong maximum principle for the Hermitian Curvature Flow (HCF), new results on the interplay of the HCF and the torsion-twisted connection, and observations on the geometry of the torsion-twisted connection on a general Hermitian manifold.

	\end{abstract}
	\maketitle
	
	\nnfootnote{2010 Mathematics Subject Classification: Primary 53C44, Secondary 53C55\\ keywords: \emph{Hermitian curvature flow}, \emph{Griffiths semipositivity}, \emph{strong maximum principle}}
	
	\section{Introduction \& Background}\label{sec:intro}

	A unifying principle in algebraic, differential and complex geometry states that the positivity of the curvature tensor puts strong topological and geometrical restrictions on the underlying space. Let us list few instances of this idea.
	
	In 1980, Siu and Yau~\cite{si-ya-80} resolved Frankel conjecture~\cite{fr-61} by proving that any K\"ahler manifold of positive \emph{holomorphic bisectional curvature} (HBC) is isomorphic (as a complex manifold) to a projective space $\C P^n$. About the same time, Mori~\cite{mo-79} approached a similar uniformization problem from the algebraic point of view. He proved a slightly stronger result that any complex projective manifold with \emph{ample} tangent bundle is isomorphic to $\C P^n$. Once the situation with strictly positive HBC was resolved it was natural to seek for a uniformization of K\"ahler manifolds of \emph{semipositive} curvature. 
	
	Splitting theorem of Cheeger and Gromoll~\cite{ch-gr-71} applied to a K\"ahler manifold $M$ of semipositive HBC implies that the universal cover of $M$ splits isometrically as $\C^k\times M^*$, where $\C^k$ is flat, and $M^*$ has semipositive HBC and Ricci form positive at least at one point. In 1979, Howard, Smyth and Wu~\cite{ho-sm-wu-79} proved a refined splitting theorem for manifolds with semipositive HBC. The key part of this theorem states that the universal cover of any such manifold is isometric to the product
	\begin{equation}\label{eq:ho-sm-wu}
	\C^k\times M_1\times\dots\times M_r,
	\end{equation}
	where each $M_i$ is a compact K\"ahler manifold of semipositive HBC with $b_2(M_i)=1$, and Ricci curvature positive at least at one point. This result could be thought of as an improved Cheeger-Gromoll splitting~\cite{ch-gr-71} under the assumption of HBC semipositivity. The main consequence of the Howard-Smyth-Wu theorem is that the classification of compact K\"ahler manifolds of semipositive HBC reduces to the study of the factors $M_i$ in~\eqref{eq:ho-sm-wu}. In~\cite{mo-88} Mok resolved the semipositive version of Frankel conjecture by describing the essential pieces in the Howard-Smyth-Wu splitting.
	\begin{theorem}[{Mok \cite[Theorem\,1]{mo-88}}]\label{thm:mok}
		Let $(M,g)$ be a compact K\'ahler manifold of semipositive holomorphic bisectional curvature such that the Ricci curvature is positive at one point. Suppose the second Betti number of $M$ is equal to one. Then either $M$ is biholomorphic to the complex projective space or $(M,g)$ is isometrically biholomorphic to an irreducible compact Hermitian symmetric manifold of rank $>2$.
	\end{theorem}

	In early 1990s in a series of papers~\cite{ca-pe-91,dps-2,dps-94} Campana, Demailly, Peternell and Schneider initiated a study of projective and complex manifolds with numerically effective (nef) tangent bundles. This is a semipositivity notion in algebraic geometry, which provides a suitable generalization of ampleness.  The authors proved several structure results for K\"ahler manifolds with nef tangent bundle and proposed the following uniformization conjecture.
	
	\begin{conjecture}[Campana-Peternell conjecture~\cite{ca-pe-91}]\label{conj:cp}
		Let $M$ be a projective manifolds with nef tangent bundle and ample anticanonical bundle $-K_M$ (i.e., $M$ is Fano). Then $M$ is isomorphic to a rational homogeneous space $G/P$, where $G$ is an algebraic reductive group, and $P$ is its parabolic subgroup.
	\end{conjecture}
	
	This conjecture can be thought of as an \emph{algebraic} version of Mok's Theorem~\ref{thm:mok}. There is much evidence supporting this conjecture: it is known in dimensions $2$ and $3$ by virtue of classification of Fano manifolds~\cite{dps-94}, it is known under various additional geometric conditions~\cite{mo-02,mo-08,ya-17}. See also a detailed survey~\cite{mu-15}.
	
	By an elementary observation made in~\cite{dps-94}, we know that all complex homogeneous manifolds admit an \emph{Hermitian} metric of semipositive Griffiths curvature (see Definition~\ref{def:gr_semipositive} and Example~\ref{ex:submersion} below). It is well known that the existence of such metric implies nefness. Motivated by this observation, in~\cite{us-16} we proposed a weak \emph{Hermitian} version of Campana-Peternell conjecture (see also~\cite{us-17-2,us-th}).
	\begin{conjecture}[Weak Hermitian Campana-Peternell conjecture]\label{conj:wcp}
		Let $(M,g)$ be a Hermitian manifold such that its Chern curvature $\Omega$ is Griffiths semipositive, and its first Chern-Ricci form $\rho$ is strictly positive. Then $M$ is isomorphic to a rational homogeneous space $G/P$, where $G$ is an algebraic reductive group, and $P$ is its parabolic subgroup.
	\end{conjecture}
	The assumptions on $(M,g)$ in Conjecture~\ref{conj:wcp} immediately imply the assumptions of Campana-Peternell conjecture, thus we refer to it as the \emph{weak} Hermitian version of Conjecture~\ref{conj:cp}. In the full generality both Campana-Peternell conjecture and its Hermitian counterpart are widely open. There is even no general result guaranteeing the existence of a nontrivial holomorphic vector field on such manifolds.
	
	Examples of~\cite{dps-94} demonstrate that the assumptions of ampleness of $-K_M$ and positivity of $\rho$ are essential for Conjectures~\ref{conj:cp} and~\ref{conj:wcp} respectively, as it was crucial for Theorem~\ref{thm:mok}. More specifically, there are examples of compact Hermitian manifolds with semipositive Griffiths curvature, which are not homogeneous. Namely, Example 3.3 in~\cite{dps-94} provides a non-homogeneous unitary-flat quotient of $\Gamma\times \Gamma\times \Gamma$, where $\Gamma$ is an elliptic curve. Recently, Yang~\cite{ya-17} observed that the metrics of Gauduchon-Ornea~\cite{ga-or-98} on the diagonal Hopf surfaces are of semipositive Griffiths curvature, while it is known that a generic diagonal Hopf surface is non-homogeneous. Both these (counter)examples could be easily included into infinite families. Interestingly, all such known examples $(M,g)$ share a very unique property revealing certain traces of homogeneity: the universal cover of $M$ admits a nontrivial, almost free (i.e. with discrete isotropy subgroups), holomorphic, isometric action of a complex Lie group. Inspired by this observation, we address the following question.
	
	\begin{question}\label{q:main}
		What can be said about a Hermitian manifold $(M,g)$ with semipositive Griffiths curvature, if we do not assume the strict positivity of the Chern-Ricci form $\rho$?
	\end{question}
	
	A satisfactory answer to Question~\ref{q:main} should provide a Hermitian analog of the Howard-Smyth-Wu splitting theorem.
	The main difficulty in approaching this question is that there is no analog of Cheeger-Gromoll splitting in Hermitian setting to make even the first step in the Howard-Smyth-Wu's proof. In this paper, we show that, while there is no hope for the actual splitting, any Hermitian manifold of semipositive Griffiths curvature is \textit{locally} modeled on a Hermitian manifold with an almost free, holomorphic and isometric action of a complex Lie group $G$. In the K\"ahler setting, the role of group $G$ was played by the factor $\C^k$.
	
	\begin{theorem*}[Theorem~\ref{thm:main}]
		Let $(M,g_0)$ be a compact Hermitian manifold with semipositive Griffiths curvature. Then there exists a complex Lie group $G$ acting almost freely, holomorphically and $g_0$-isometrically on the universal cover $\widetilde{M}$ of $M$ with the following property:
		\smallskip
		
		\noindent
		For any $\epsilon>0$, $k\in \N$ there exists a $G$-invariant Hermitian metric $g$ on $M$ such that
		\begin{enumerate}
			\item[(i)] $||g_0-g||_{C^k,g_0}< \epsilon$;
			
			\item[(ii)] $(M, g)$ has semipositive Griffiths curvature;
			
			\item[(iii)] the distribution generated by the infinitesimal action of $G$ on $\widetilde{M}$ coincides with the null space distribution of the Chern-Ricci form $\rho=\rho(g)$:
			\[\mathrm{Null}(\rho_x)=T_x (G.x),\]
			where $G.x$ is the orbit of $x\in\widetilde{M}$.
		\end{enumerate}
	\end{theorem*}

	Let us briefly outline the strategy of the proof of Theorem~\ref{thm:main}. In~\cite{us-16,us-17-2} we proposed a program of attacking the weak Hermitian Campana-Peternell conjecture by the means of metric flows. We identified a member of a general family of Hermitian Curvature Flows~\cite{st-ti-11}, which preserves and improves various curvature semipositivity notions in Hermitian geometry. It was observed in~\cite{us-16,us-17-2,us-th} that geometric properties of this flow are closely related to the torsion-twisted connection $\n^T$ (see Definition~\ref{def:torsion-twisted}). Our proof of Theorem~\ref{thm:main} combines several ingredients, which are of independent interest:
	\begin{enumerate}
		\item[1.] Structure result on the $\n^T$-parallel subspaces of the tangent space on a general Hermitian manifold (Proposition~\ref{prop:killing}, Proposition~\ref{prop:integrable_distribution}, Theorem~\ref{thm:G_action});
		\item[2.] Invariance of the torsion-twisted holonomy under the HCF on a general Hermitian manifold (Theorems~\ref{thm:nt_invariance} and~\ref{thm:nt_invariance_2});
		\item[3.] Parabolic maximum principle for the Chern curvature on a Hermitian manifold with semipositive Griffiths curvature (Theorem 5.2 of~\cite{us-16}).
	\end{enumerate}
	Given a compact Hermitian manifold $(M,g_0)$ with semipositive Griffiths curvature, we: (i) run HCF $g(t)$ for arbitrary small time; (ii) for time slice $t>0$ use regularization properties of HCF to construct a $\n^T$-fixed subspace $\mc F\subset T^{1,0}M$; (iii) use geometric properties of the torsion-twisted connection to identify $\mc F$ with an infinitesimal isometric action of a complex Lie group; (iv) use invariance of the torsion-twisted holonomy under HCF to extend this $G$-structure to the $t=0$ time slice.

	The rest of this paper is organized as follows. In Section~\ref{sec:notations} we provide background on Hermitian geometry and set up notations. In Section~\ref{sec:torsion-twisted-connection}, we study the geometry of the torsion-twisted connection on a general Hermitian manifold. Section~\ref{sec:hcf} describes the interplay between the Hermitian Curvature Flow and the torsion-twisted connection. Finally, in Section~\ref{sec:main} we give the proof of our main result.

	\section{Notations}\label{sec:notations}
	
	Let $(M,g)$ be a Hermitian manifold, i.e., a complex manifold $M$ with an operator of almost complex structure $J$ and a $J$-invariant Riemannian metric $g$. Operator of almost complex structure $J$ induces the decomposition of the complexified tangent spaces into $\pm\sqrt{-1}$-eigenspaces:
	\[
	T_\C M:=TM\otimes\C=T^{1,0}M\oplus T^{0,1}M,
	\]
	and the corresponding decomposition of the spaces of exterior forms:
	\[
	\Lambda^r_\C(M):=\Lambda^r(T^*M)\otimes\C=\bigoplus_{p+q=r}\Lambda^{p,q}(M).
	\]
	Any $J$-compatible connection on $TM$ preserves this decompositions. In this paper we will be using two connections canonically attached to a Hermitian manifold: the Chern connection $\n$ and the \emph{torsion-twisted} connection $\n^T$. The latter will be defined and studied in the next section.
	
	\begin{definition}[Chern connection]
		Chern connection $\n$ is the unique \emph{unitary} connection on $TM$ which is compatible with the holomorphic structure, i.e., $\n^{0,1}=\bar{\pd}$ on $\Gamma(T^{1,0}M)$.
	\end{definition}
	
	\begin{remark}
		The defining properties of the Chern connection immediately imply that the torsion tensor $T\in\Lambda^2(M)\otimes TM$,
		\[
		T(X,Y):=\n_X Y-\n_Y X-[X,Y]
		\]
		has vanishing $\Lambda^{1,1}(M)\otimes TM$ component. In other words, for any $\xi,\eta\in T^{1,0}M$, we have $T(\xi,\bar{\eta})=0$. Equivalently, for any $X,Y\in TM$
		\[
		T(X,JY)=T(JX,Y)=JT(X,Y).
		\]
	\end{remark}
	
	\begin{definition}
		The \emph{Chern curvature} of a Hermitian manifold $(M, g)$ is the curvature of $\n$.
		\[
		\Omega(X,Y)Z:=(\n_X\circ\n_Y-\n_Y\circ\n_X-\n_{[X,Y]})Z,
		\]
		where $X,Y,Z\in \Gamma(TM)$.
	\end{definition}
	The Chern curvature tensor is completely determined by its components $\Omega_{i\bar j k}^{\ \ \ l}$
	\[
	\Omega_{i\bar j k}^{\ \ \ l}\frac{\pd}{\pd z^l}:=\Omega\Bigl(\frac{\pd}{\pd z^i},\frac{\pd}{\pd \bar{z^j}}\Bigr)\frac{\pd}{\pd z^k}.
	\]

	Unlike the Riemannian case, the Chern curvature does not satisfy the classical Bianchi identities, since the Chern connection has torsion. However, in this case, slightly modified identities, involving torsion still hold~\cite[Ch.\,III, Thm.\,5.3]{ko-no-63}.
	
	\begin{proposition}[Bianchi identities for the Chern curvature and torsion]
		\begin{equation}
		\begin{split}
		\curvature_{i \bar{j} k \bar{l}}=\curvature_{k \bar{j} i \bar{l}} + \n_{\bar{j}} T_{ki\bar l},\quad &
		\curvature_{i \bar{j} k \bar{l}}=\curvature_{i \bar{l} k \bar{j}} + \n_{i} T_{\bar{l}\bar{j} k},\\
		\n_i T_{jk}^l+\n_k T_{ij}^l+\n_j T_{ki}^l=&
		T_{ij}^pT_{kp}^l+T_{jk}^pT_{ip}^l+T_{ki}^pT_{jp}^l.\\
		\n_m \curvature_{i \bar{j} k \bar{l}} = \n_i\curvature_{m \bar{j} k \bar{l}} + T^p_{i m}\curvature_{p \bar{j} k \bar{l}},\quad &
		\n_{\bar n} \curvature_{i \bar{j} k \bar{l}} = \n_{\bar j}\curvature_{i \bar{n} k \bar{l}} + T^{\bar s}_{\bar j \bar n}\curvature_{i \bar{s} k \bar{l}}.
		\end{split}
		\end{equation}
	\end{proposition}
	
	\begin{definition}\label{def:chern_ricci}
		The \emph{Chern-Ricci form} of $(M,g)$ is the curvature 2-form $\rho\in\Lambda^{1,1}(M)$ of the line bundle $\Lambda^{n}(T^{1,0}M)$ equipped with the metric $g$. Form $\rho$ is closed and represents class $2\pi c_1(M)$ in $H^2(M,\R)$. In coordinates, $\rho$ is represented by a contraction of the Chern curvature: $\rho_{i\bar j}=\Omega_{i\bar j k}^{\ \ \ k}$.
	\end{definition}
	
	Next, we define the key \emph{curvature semipositivity} notion for a Hermitian manifold $(M,g)$~--- semipositivity in the sense of Griffiths. One might think of it as an analogue of semipositive sectional curvature in the Riemannian setting, or, semipositivity of HBC in the K\"ahler setting.
	
	\begin{definition}\label{def:gr_semipositive}
		Hermitian manifold $(M,g)$ is \emph{Griffiths semipositive} if its Chern curvature satisfies
		\[
		\Omega(\xi,\bar{\xi},\eta,\bar{\eta})\ge 0
		\]
		for any $\xi,\eta\in T^{1,0}M$. In what follows, we will write simply $\Omega\ge_{\mathrm{Gr}} 0$.
	\end{definition}

	\section{Geometry of the Torsion-Twisted Connection}\label{sec:torsion-twisted-connection}
	
	In this section we define and study geometric properties of the torsion-twisted connection $\n^T$. This connection was introduced in~\cite{us-16} as a tool for studying maximum principle for tensors under Hermitian Curvature Flow. Below we prove that the maximal $\n^T$-parallel subbundle $\mc F\subset T^{1,0}M$ defines an infinitesimal holomorphic, isometric action of a complex Lie group.
		
	\begin{definition}[Torsion-twisted connection]\label{def:torsion-twisted}
		Let $(M,g)$ be a Hermitian manifold, and denote by $\n$ its Chern connection. \emph{Torsion-twisted connection} $\n^T$ is defined by the identity:
		\[
		\n^T_X Y=\n_X Y-T(X,Y),
		\]
		where $T\in\Lambda^2(M)\otimes TM$ is the torsion of the Chern connection. We extend $\n^T$ to the connection on $T_\C M$ and all associated tensor bundles in an obvious way.
	\end{definition}
	
	\begin{remark}
		The definition of the Chern connection implies that the (1,1)-type part of the torsion tensor $T$ vanishes, therefore $\n^TJ=0$, and $(\n^T)^{0,1}=\n^{0,1}=\bar{\pd}$ on $\Gamma(T^{1,0}M)$. It is also clear that for $X,Y\in \Gamma(TM)$, we have
		\[
		\n^T_X Y=\n_Y X+[X,Y].
		\]
	\end{remark}
	
	The following elementary proposition highlights the relevance of the torsion-twisted connection for the study of Hermitian geometry on $(M,g)$.
	\begin{proposition}\label{prop:killing}
		Vector field $\zeta\in\Gamma(T^{1,0}M)$ is $\n^T$-parallel:
		\[
		\n^T \zeta=0,
		\]
		if and only if $\zeta$ is a holomorphic Killing vector field.
	\end{proposition}
	\begin{proof}
		Since $\n^T$ is compatible with the holomorphic structure on $T^{1,0}M$, the vanishing of $(\n^T)^{0,1} \zeta$ is equivalent to the fact that $\zeta$ is holomorphic.
		
		Now let us prove that holomorphic vector field $\zeta$ is Killing if and only if $(\n^T)^{1,0}\zeta=0$. For $\xi,\eta\in \Gamma(T^{1,0}M)$ we have
		\begin{equation}
		\begin{split}
		(\mc L_\zeta g)(\xi,\bar{\eta})
			&=\zeta\cdot{g(\xi,\bar{\eta})}-g([\zeta,\xi],\bar{\eta})-g(\xi,[\zeta,\bar{\eta}])\\
			&=g(\n_\zeta \xi,\bar{\eta})+g(\xi,\n_{\zeta} {\bar{\eta}})-g([\zeta,\xi],\bar{\eta})-g(\xi,[\zeta,\bar{\eta}])\\
			&=g(\n^T_\xi \zeta,\bar{\eta})+g(\xi,\n^T_{\bar{\eta}}\zeta)=g(\n^T_\xi \zeta,\bar{\eta}).
		\end{split}
		\end{equation}
		Hence, holomorphic vector field $\zeta$ is Killing if and only if $(\n^T)^{1,0}\zeta=0$.
	\end{proof}
	
	Let $\Omega^T$ be the curvature of the torsion-twisted connection on $T^{1,0}M$. In what follows, we will need an explicit expression for $\Omega^T$ through $\Omega$ and $T$.
	
	\begin{proposition}\label{prop:torsion_twisted_curvature}
		Torsion-twisted curvature $\curvature^T\in \Lambda^2(M)\otimes\End(T^{1,0}M)$ is of the complex type $(2,0)+(1,1)$ and has components
		\begin{equation}\label{eq:torsion_twisted_curvature}
		\begin{split}
		\Omega^T(\xi,\bar\eta)\zeta&=\curvature(\zeta,\bar{\eta})\xi,\\
		\Omega^T(\xi,\eta)\zeta&=\n_\zeta T(\xi,\eta),
		\end{split}
		\end{equation}
		where $\xi,\eta,\zeta\in T^{1,0}M$.
	\end{proposition}
	\begin{proof}
		The torsion-twisted connection is compatible with the holomorphic structure, therefore the (0,2)-type part of curvature vanishes: ${(\curvature^T)}^{0,2}=\bar{\pd}^2=0$.
		
		For a vector $w\in T_\C M$, define an endomorphism
		\[
		T_w\colon T_\C M\to T_\C M,\quad v\mapsto T(w,v).
		\]
		Since the (1,1)-type part of $T$ vanishes, for $\xi\in T^{1,0}M$, the endomorphism $T_\xi$ is zero on $T^{0,1}M$ and maps $T^{1,0}M$ into $T^{1,0}M$. Now, assume that $\xi,\eta,\zeta$ are local coordinate holomorphic vector fields. Then for the type (1,1) part we have
		\begin{equation}
		\begin{split}
		\curvature_{\n^T}(\xi,\bar{\eta})\zeta{=}&
		[\n_{\xi}-T_\xi,\n_{\bar\eta}-T_{\bar\eta}]\zeta=
		\bigl(
		\curvature(\xi,\bar{\eta})-[\n_\xi,T_{\bar{\eta}}]+[\n_{\bar{\eta}},T_\xi]+[T_\xi,T_{\bar\eta}]
		\bigr)\zeta
		\\{=}&
		\curvature(\xi,\bar{\eta})\zeta+\n_{\bar{\eta}}(T_\xi\zeta)
		{=}
		\curvature(\xi,\bar{\eta})\zeta+(\n_{\bar\eta} T)(\xi,\zeta),
		\end{split}
		\end{equation}
		where in the third equality we use that (i) $[\xi,\bar{\eta}]=0$; (ii) $\n_{\bar{\eta}}$ annihilates any holomorphic vector field, (iii) $T_{\bar{\eta}}$ annihilates any (1,0) vector. By the first Bianchi identity, the final expression equals $\curvature(\zeta,\bar{\eta})\xi$.
		
		Similarly, we compute the (2,0) part of the curvature $\curvature_{\n^T}$, using the fact that the (2,0) part of the Chern curvature vanishes.
		\begin{equation}
		\begin{split}
		\curvature_{\n^T}(\xi,\eta)\zeta&=
		[\n_{\xi}-T_{\xi},\n_{\eta}-T_{\eta}]\zeta=
		\bigl(
		[\n_\xi,\n_\eta]+[\n_\eta,T_\xi]-[\n_\xi,T_\eta]+[T_\xi,T_\eta]
		\bigr)\zeta\\&=
		\n_{\eta}(T(\xi,\zeta))-T(\xi,\n_\eta\zeta)-
		\n_\xi(T(\eta,\zeta))+T(\eta,\n_\xi\zeta)+
		T(\xi,T(\eta,\zeta))-T(\eta,T(\xi,\zeta))\\&=
		\n_\xi T(\zeta,\eta)+\n_\eta T(\xi,\zeta)+T(\n_\eta\xi-\n_\xi\eta,\zeta)+T(\xi,T(\eta,\zeta))-T(\eta,T(\xi,\zeta))\\&=
		\n_\xi T(\zeta,\eta)+\n_\eta T(\xi,\zeta)+T(T(\eta,\xi),\zeta)+T(T(\zeta,\eta),\xi)+T(T(\xi,\zeta),\eta)=\n_{\zeta}T(\xi,\eta),
		\end{split}
		\end{equation}
		where in the last identity we use the (3,0) type part of the first Bianchi identity.
	\end{proof}
	
	For a base point $x\in M$, we denote by 
	\[
	\mathrm{Hol}_x(\n^T)\subset GL(T_x^{1,0}M)\simeq GL_n(\C)
	\]
	the \emph{holonomy group} of $\n^T$, that is the group generated by $\n^T$-parallel transport along all smooth loops in $M$ based at $x\in M$. Let $\mathrm{Hol}_x^0(\n^T)$ be the \emph{restricted holonomy group} of $\n^T$, that is generated by parallel transport along \emph{contractible} loops. Clearly $\mathrm{Hol}_x(\n^T)=\mathrm{Hol}_x^0(\n^T)$ if $M$ is simply-connected.
	
	Finally, let $\mf{hol}_x(\n^T)\subset \End(T_x^{1,0}M)$ be the Lie algebra of $\mathrm{Hol}_x^0(\n^T)$.
	The classical Ambrose-Singer theorem relates the holonomy Lie algebra and the curvature tensor.
	\begin{theorem}[Ambrose-Singer \cite{am-si-53}]\label{thm:ambrose-singer}
		The holonomy Lie algebra $\mf{hol}_x(\n^T)$ is spanned by all the elements of $\End(T_x^{1,0}M)$ of the form
		\[
		\mc P_\gamma^{-1}\circ\Omega_y^T(X,Y) \circ \mc P_\gamma,
		\]
		where $y\in M$, $X,Y\in T_yM$, $\gamma$ is a smooth path from $x$ to $y$ in $M$, and
		\[
		\mc P_\gamma\colon T_x^{1,0}M\to T_y^{1,0}M
		\]
		is the parallel transport map.
	\end{theorem}

	Define $\mc F_x\subset T_x^{1,0}M$
	\[
	\mc F_x:=\{\xi\in T_x^{1,0}M\ |\ \mathrm{Hol}_x(\n^T)\xi=\xi\}
	\]
	to be the set of all vectors in $T_x^{1,0}M$ fixed by $\mathrm{Hol}_x(\n^T)$. Clearly, $\mc F=\bigcup_x\mc F_x\subset T^{1,0}M$ is a subbundle of $T^{1,0}M$ invariant under the $\n^T$-parallel transport. Moreover
	\begin{enumerate}
		\item any $\xi \in \mc F_x$ extends to a $\n^T$-parallel global vector field $\xi\in\Gamma(T^{1,0}M)$;
		\item any $\n^T$-parallel vector field $\xi\in\Gamma(T^{1,0}M)$ is a section of $\mc F$.
	\end{enumerate}
	 Proposition~\ref{prop:killing} implies that $\mc F$ is a holomorphic subbundle of $T^{1,0}M$. Now we prove that $\mc F$ has extra structure.
	\begin{proposition}\label{prop:integrable_distribution}
		Holomorphic distribution $\mc F\subset T^{1,0}M$ is Frobenius-integrable, i.e.,
		\[
		[\mc F,\mc F]\subset \mc F.
		\]
		Moreover, if $\xi_1,\dots,\xi_r$ is a $\n^T$-parallel basis of $\mc F$, then
		\[
		[\xi_i,\xi_j]=c_{ij}^k\xi_k,
		\]
		with the structure constants $c_{ij}^k=c_{ij}^k(x)$ independent of $x\in M$.
	\end{proposition}
	\begin{proof}
		Let $\xi,\eta\in\Gamma(\mc F)$ be two $\n^T$-parallel sections of $\mc F$. We claim that
		\[
		\n^T[\xi,\eta]=0.
		\]
		This will prove the proposition.
		
		First, observe that, since $\n^T_\xi\eta=\n^T_\eta\xi=0$, there is an identity
		\[
			[\xi,\eta]=-T(\xi,\eta).
		\]
		We will prove that $\n^T(T(\xi,\eta))=0$ by computing the covariant derivative separately in (1,0) and (0,1)-type directions.
		\begin{enumerate}
		\item Let $\zeta$ be a (1,0)-vector. Then using the fact that $\xi$ and $\eta$ are $\n^T$-parallel, and applying the (3,0)-part of the first Bianchi identity, we compute:
		\begin{equation}
		\begin{split}
		\n^T_{\zeta}(T(\xi,\eta))&
			=\n_\zeta(T(\xi,\eta))
				-T(\zeta,T(\xi,\eta))\\&
			=(\n_\zeta T)(\xi,\eta)
				+T(\n_\zeta\xi,\eta)
				+T(\xi,\n_\zeta\eta)
				-T(\zeta,T(\xi,\eta))\\&
			=(\n_\zeta T)(\xi,\eta)
				+T(T(\zeta,\xi),\eta)
				+T(T(\eta,\zeta),\xi)
				+T(T(\xi,\eta),\zeta)\\&
			=(\n_\xi T)(\zeta,\eta)
				+(\n_\eta T)(\xi,\zeta).
		\end{split}
		\end{equation}
		Both summands in the final expression vanish. Indeed, by Proposition~\ref{prop:torsion_twisted_curvature} we have
		\[
			(\n_\xi T)(\zeta,\eta)=\Omega^T(\zeta,\eta)\xi,
		\]
		and, by Ambrose-Singer theorem, $\Omega^T(\zeta,\eta)\in\mf{hol}_x(\n^T)\otimes\C$. Vector field $\xi$ is $\n^T$-parallel, hence is annihilated by the holonomy Lie algebra, and $(\n_\xi T)(\zeta,\eta)=0$.
		
		\item Let $\bar{\zeta}$ be a $(0,1)$-vector. Similarly to the above, using the fact that $T^{1,1}=0$, and applying the (2,1)-type part of the first Bianchi identity we find:
		\begin{equation}
		\begin{split}
		\n^T_{\bar\zeta}(T(\xi,\eta))&
		=\n_{\bar\zeta}(T(\xi,\eta))
		=(\n_{\bar\zeta}T)(\xi,\eta)\\&
		=\Omega(\eta,\bar{\zeta})\xi-\Omega(\xi,\bar{\zeta})\eta\\&
		=\Omega^T(\xi,\bar{\zeta})\eta-\Omega^T(\eta,\bar{\zeta})\xi.
		\end{split}
		\end{equation}
		For the same reason, as in the first part, both terms in the final expression vanish.
		\end{enumerate}
	\end{proof}
	
	Proposition~\ref{prop:integrable_distribution} implies that $\mc F$ is spanned by $\n^T$-parallel vector fields, which form a finite-dimensional complex Lie algebra $\mf g$ of dimension $\dim_\C\mf g=\dim_\C\mc F_x$. The injective homomorphism $\mf g\to \Gamma(T^{1,0}M)$ induces an infinitesimal action of the corresponding complex Lie group $G$ with discrete isotropy subgroups (such an action is called \emph{almost free}). With this data we can reformulate Proposition~\ref{prop:integrable_distribution} as follows.
	\begin{theorem}\label{thm:G_action}
		Let $(M,g)$ be a Hermitian manifold with a holomorphic distribution $\mc F\subset T^{1,0}M$
		\[
		\mc F_x:=\{\xi\in T_x^{1,0}M\ |\ \mathrm{Hol}_x(\n^T)\xi=\xi\}.
		\]
		Then there exists a complex Lie group $G$ acting almost freely, holomorphically and isometrically on~$(M,g)$ such that at any $x\in M$
		\[
		T_x^{1,0}(G_\cdot x)=\mc F_x,
		\]
		where $G_\cdot x$ is the orbit of $x$. In particular, distribution $\mc F$ is integrable, and the leaves of the foliation generated by $\mc F$ coincide with the orbits of $G$.
	\end{theorem}
	
	The following example illustrates that any complex group $K$ can appear as group $G$ in Theorem~\ref{thm:G_action}.
	
	\begin{example}\label{ex:homogeneous}
		Consider a complex Lie group $M=K$. Let
		\begin{itemize}
			\item $\{\xi_1,\dots,\xi_n\}$ be a basis of right-invariant holomorphic vector fields (inducing the action of $K$ on itself by multiplication from the \emph{left});
			\item $\{\eta_1,\dots,\eta_n\}$ be a basis of left-invariant holomorphic vector fields (inducing the action of $K$ on itself by multiplication from the \emph{right}).
		\end{itemize}
		Equip $K$ with a right-invariant metric $g$:
		\[
		g(\xi_i,\bar{\xi_j})=g_{i\bar j},
		\]
		where $g_{i\bar j}=g_{i\bar j}(x)$ are constant functions on $M$. For the Chern connection $\n$ on $(K,g)$ we have:
		\[
		0=\zeta\cdot g(\xi_i,\bar{\xi_j})=g(\n_{\zeta}\xi_i,\bar{\xi_j}),\quad\zeta\in T^{1,0}K,
		\]
		therefore $\n\xi_i=0$. Moreover, as the left and right actions of $K$ on itself commute, we have $[\xi_i,\eta_j]=0$. Hence
		\[
		\n^T_{\xi_j}\eta_i=\n_{\eta_i}\xi_j+[\xi_j,\eta_i]=0,
		\]
		so $\{\eta_1,\dots,\eta_n\}$ is a $\n^T$-parallel basis of $T^{1,0}K$. Therefore, for a complex Lie group with a right-invariant metric, $(K,g)$, we have that
		\[
		\mc F=T^{1,0}K,
		\]
		and the group $G$ of Theorem~\ref{thm:G_action} is the group $K$ acting holomorphically and isometrically on itself by multiplication from the right.
		
		Alternatively, one can see that vector fields $\eta_i$, $1\le i\le n$ are $\n^T$-parallel directly from the only if part of Proposition~\ref{prop:killing}, since each $\eta_i$ is holomorphic and Killing by our choice of metric.
	\end{example}
	
	\begin{remark}
		It is instructive to compare Theorem~\ref{thm:G_action} with the corresponding statement about the Levi-Civita connection in the standard (real) Riemannian setting. In that case, $M$ splits isometrically as $M=F\times N$, where $F$ is flat with $\mc F\subset TM$ corresponding to $TF$.
	
		Going back to the Hermitian setting, let $\mc F^\perp$ be the orthogonal complement of $\mc F$. Unlike the Riemannian case, $\mc F^\perp$ is not necessarily integrable.
		
		In general, the holomorphic distribution $\mc F\subset T^{1,0}M$ does not come from (even local) splitting $M\simeq F\times N$, but rather defines a non-singular foliation of $M$ by the orbits of $G$. Locally, at a point $x\in M$ this foliation is modeled by a bundle-like Hermitian submersion
		\[
		\mc U(e; G)\to \mc U(x,M)\to\mc V, 
		\]
		where $\mc U(e; G)$ is a neighbourhood of $e$ in $G$, $\mc U(x; M)$ is an appropriate neighbourhood of $x$ in $M$, and $\mc V$ is a ``local orbit space''.
	\end{remark}
	
	Example~\ref{ex:homogeneous} is trivial from the point of view of the foliation generated by distribution $\mc F$, since $\mc F=T^{1,0}M$ is the whole tangent space, and there is a single $G$-orbit. Now, we construct a more interesting example, by examining torus-invariant metrics on a primary Hopf surface of class 1.
	\begin{example}[Hopf surfaces]
		Consider a primary Hopf surface of class 1:
		\[
		M_{a_1,a_2}:=\bigl(\C^2\backslash\{(0,0)\}\bigr)/\la\gamma\ra,\quad \gamma\colon (z_1,z_2)\mapsto (e^{a_1}z_1,e^{a_2}z_2).
		\]
		where $a_1,a_2\in\C$, $\Re (a_1)\ge\Re(a_2)>0$. 
		There is a 3-dimensional torus $T^3=(S^1)^2\times S^1$ acting on $M=M_{a_1,a_2}$ by biholomorphisms, where the action of $(S^1)^2\simeq (U(1))^2\subset (\C^\times)^2$ is induced by the coordinate-wise multiplication
		\begin{equation}
		\begin{split}
		(S^1)^2\times M
		&\to
		M\\
		(u_1, u_2)\times (z_1,z_2)
		&\mapsto
		(u_1 z_1, u_2 z_2),
		\end{split}
		\end{equation}
		and the action of $S^1$ is given by
		\begin{equation}
		\begin{split}
		S^1\times M
		&\to
		M\\
		e^{2\pi it}\times (z_1,z_2)
		&\mapsto
		(e^{a_1t} z_1, e^{a_2t} z_2),
		\end{split}
		\end{equation}
		
		At a point $x\in M$, let $\mf t_x\subset T_xM$ be the real 3-dimensional subspace generated by the infinitesimal action of $T^3$. Then holomorphic vector field
		\[
		\zeta:=
		\Re(a_1)z_1\frac{\pd}{\pd z_1}
		+
		\Re(a_2)z_2\frac{\pd}{\pd z_2}
		\]
		generates a complex 1-dimensional distribution $\mc F\subset T^{1,0}M$ which corresponds under the natural identification $T^{1,0}M\simeq TM$ to the $J$-invariant subspace $\mf t_x\cap J\mf t_x\subset TM$.
		
		Now let $g$ be any Hermitian metric on $M$. After averaging $g$ with respect to the $T^3$-action, we can assume that $g$ is $T^3$-invariant. Then $\zeta$ is a holomorphic Killing vector field on $(M,g)$, and by Proposition~\ref{prop:killing}, we conclude that $\zeta$ is $\n^T$-parallel. We claim that $\mc F$ is spanned by $\zeta$:
		\[
		\mc F:=\{\xi\in T_x^{1,0}M\ |\ \mathrm{Hol}_x(\n^T)\xi=\xi\}=\C\cdot\zeta
		\]
		Indeed, if there is another linearly independent $\n^T$-parallel vector field, then $T^{1,0}M$ would be holomorphically trivial. This is not the case, since $T^{1,0}M$ admits a not everywhere nonzero holomorphic section, e.g., $z_1\pd/\pd z_1$. 
		
		Therefore the distribution $\mc F$ of Theorem~\ref{thm:G_action} coincides with $\C\cdot\zeta\subset T^{1,0}M$, and the group $G$ is isomorphic to $\C$ acting on $M_{a_1,a_2}$ via
		\begin{equation}
		\begin{split}
		\mathrm{ev}\colon\C\times M&\to M\\
		u\times (z_1,z_2)&\mapsto(e^{\Re(a_1) u}z_1,e^{\Re(a_2) u}z_2).
		\end{split}
		\end{equation}
		Depending on the moduli parameters $a_1,a_2\in \C$, a generic orbit of $G\simeq\C$ is either closed and isomorphic to $(S^1)^2$, or has $T^3$-orbit as its closure, and is isomorphic to one of $\C^\times$ and~$\C$. If the generic orbit is $T^2$, the Hopf surface is the total space of a Seifert fibration over $\C P^1$.
	\end{example}

	\section{Hermitian Curvature Flow}\label{sec:hcf}
	Let $(M,g)$ be a Hermitian manifold of complex dimension $\dim_\C M=n$. Consider an evolution equation for the Hermitian metric $g=g(t)$.
	\begin{equation}\label{eq:HCF}
	\pd_t g_{i\bar j}=-\cric_{i\bar j}-\mc Q_{i\bar j},
	\end{equation}
	where $\cric_{i\bar j}=g^{m\bar n}\curvature_{m\bar ni\bar j}$ is the \emph{second Chern-Ricci curvature} and $\mc Q_{i\bar j}=\frac{1}{2}g^{m\bar n}g^{p\bar s}T_{mp\bar j}T_{\bar n\bar s i}$ is a quadratic torsion term. Flow~\eqref{eq:HCF} is a member of the family of Hermitian Curvature Flows, introduced by Streets and Tian in~\cite{st-ti-11}. It is proved there that all HCFs are strictly parabolic evolution equations for $g$, and, hence, admit short-time solutions. The particular flow~\eqref{eq:HCF} was first considered by the author in~\cite{us-16} (see also~\cite{us-17,us-17-2,us-th}). Further we refer to the flow~\eqref{eq:HCF} as the HCF.
	
	Our primary motivation for considering this member of HCF family comes from the following result.
	\begin{theorem}[{\cite[Theorem~5.1]{us-16}}]\label{thm:hcf-griffiths-preserved}
		Let $g(t), t\in[0,\tau)$ be the solution to the HCF~\eqref{eq:HCF} on a compact complex Hermitian manifold $(M,g_0)$. Assume that the Chern curvature $\Omega(g_0)$ at the initial moment $t=0$ is Griffiths semipositive (resp. positive). Then for $t\in[0,\tau)$ the Chern curvature $\Omega(g(t))$ remains Griffiths semipositive (resp. positive).
	\end{theorem}
	In~\cite{us-17-2} we further strengthened Theorem~\ref{thm:hcf-griffiths-preserved} by proving that flow~\eqref{eq:HCF} preserves many other natural (semi)positivity notions in complex geometry, and satisfies generalized maximum principle for tensors.
	
	The results of~\cite{us-16} and~\cite{us-17-2} suggest that geometric features of the HCF are closely related to the properties of the underlying torsion-twisted connection. Now, we establish a further link between the HCF and $\n^T$.
	
	\begin{theorem}\label{thm:nt_invariance}
		Assume that $g(t)$, $t\in [0,\tau)$ solves the HCF on a Hermitian manifold $(M, g_0)$. Let $\mc E\to M$ be a holomorphic tensor bundle associated with $T^{1,0}M$, e.g., $T^{1,0}M$, $\End(T^{1,0}M)$, $\Lambda^{k,0}(M)$. Denote by $\n_{\mc E}(t)$ the connection induced on $\mc E$ by the torsion-twisted connection $\n^T(t)$ of $(M,g(t))$.
		
		If a holomorphic subbundle $\mc V\subset \mc E$ is invariant under $\n_{\mc E}(\tau')$ for some $\tau'\in[0,\tau)$, then it is invariant under $\n_{\mc E}(t)$ for any $t\in[0,\tau)$.
	\end{theorem}
	\begin{proof}
		To a short exact sequence of vector bundles
		\begin{equation}\label{eq:short_exact_sequence}
		0\to \mc V\xrightarrow{\iota}\mc E\xrightarrow{p}\mc Q\to 0,
		\end{equation}
		and any connection $D$ on $\mc E$, one can associate a \emph{second fundamental form}
		\[
		\pmb{\beta}\in \Lambda^1(M)\otimes\Hom(\mc V,\mc Q),\quad \pmb{\beta}_\bullet(v)=p(D_\bullet (\iota v))
		\]
		The second fundamental form $\pmb{\beta}$ vanishes if and only if $\mc V$ is invariant under $D$.
		
		For convenience, we denote the second fundamental form for $\mc V\subset \mc E$ with respect to the $t$-dependent connection $D=\n_{\mc E}(t)$ by $\pmb{\beta}=\pmb{\beta}(t)$. We know that $\pmb{\beta}(\tau')=0$, and we want to prove that $\pmb\beta(t)=0$ all $t\in[0,\tau)$. The idea is to prove that $\pmb{\beta}(t)$ satisfies a strongly parabolic equation, and to check that $\pmb{\beta}(t)\equiv 0$ is a solution. Then, due to a general uniqueness result for backward/forward solutions of parabolic PDEs, we conclude that $\pmb{\beta}(t)$ vanishes for all $t$.
		\begin{lemma}\label{lm:torsion_twisted_connection_evolution}
			Connection $\n^T=\n^T(t)$ satisfies
			\begin{equation}
			\begin{split}
			\frac{d}{dt}(\n^T)_{ij}^{\ \ k}&=
				-g^{m\bar n}(\widetilde{\n}_m\Omega^T)_{i\bar n j}^{\ \ \ k}
				-
				\frac{1}{2}g^{m\bar n}g^{p\bar s i}T_{\bar n\bar s}(\Omega^T)_{mpj}^{\ \ \ k}
			\\
			\frac{d}{dt}(\n^T)_{\bar ij}^{\ \ k}&=0,
			\end{split}
			\end{equation}
			where $\widetilde{\n}:=\n\otimes 1+1\otimes\n^T$ is a connection on $\Lambda^2(M)\otimes \End(T^{1,0}M)$.
		\end{lemma}
		\begin{proof}
			First, we note that $(\n^T)^{0,1}=\bar{\pd}$, hence the (0,1)-type part of $\n^T$ is independent of $g$.
			
			By a direct computation (see~\cite[\S 10]{st-ti-11}) we have
			\[
			\frac{d}{dt}\n_{ij}^{\ \ k}=g^{k\bar n}\n_i h_{j\bar n}
			\]
			\[
			\frac{d}{dt}T_{ij}^k=g^{k\bar n}(\n_i h_{j\bar n}-\n_j h_{i\bar k}),
			\]
			where $h_{i\bar j}=\frac{d}{dt}g_{i\bar j}$. Therefore,
			\[
			\frac{d}{dt}(\n^T)_{ij}^{\ \ k}=\frac{d}{dt}(\n-T)_{ij}^{\ \ k}=g^{k\bar n}\n_j h_{i\bar n}
			\]
			Now, recall that
			\[
			h_{i\bar j}=-(g^{m\bar n}\Omega_{m\bar n i\bar j}+\frac{1}{2}g^{m\bar n}g^{p\bar s}T_{mp\bar j}T_{\bar n\bar s i}),
			\]
			and we compute
			\begin{equation}
			\begin{split}
			\frac{d}{dt}(\n^T)_{ij}^{\ \ k}&=
			-g^{m\bar n}\n_j\Omega_{m\bar n i}^{\ \ \ k}
				-\frac{1}{2}g^{m\bar n}g^{p\bar s}\n_j(T_{mp}^kT_{\bar n\bar s i})\\&=
			-g^{m\bar n}\n_m\Omega_{j\bar n i}^{\ \ \ k}-g^{m\bar n}T_{mj}^l\Omega_{l\bar n i}^{\ \ \ k}
				-\frac{1}{2}g^{m\bar n}g^{p\bar s}(
					\n_j T_{mp}^k T_{\bar n\bar s i}+
					T_{mp}^k \n_j T_{\bar n\bar s i})\\&=
			-g^{m\bar n}\n_m\Omega_{j\bar n i}^{\ \ \ k}-
				g^{m\bar n}T_{mj}^l\Omega_{l\bar n i}^{\ \ \ k}+
				g^{m\bar n}T_{mp}^k\Omega_{j\bar n i}^{\ \ \ p}-
				\frac{1}{2}g^{m\bar n}g^{p\bar s}T_{\bar n\bar s i}\n_jT_{mp}^k\\&=
			-g^{m\bar n}(
				\n_m\Omega_{j\bar n i}^{\ \ \ k}+
				T_{mj}^l\Omega_{l\bar n i}^{\ \ \ k}-
				T_{mp}^k\Omega_{j\bar n i}^{\ \ \ p}
			)
			-\frac{1}{2}g^{m\bar n}g^{p\bar s}T_{\bar n\bar s i}\n_jT_{mp}^k\\&=
			-g^{m\bar n}(\widetilde{\n}_m\Omega^T)_{i\bar n j}^{\ \ \ k}
			-\frac{1}{2}g^{m\bar n}g^{p\bar s}T_{\bar n\bar s i}(\Omega^T)_{mpj}^{\ \ \ k},
			\end{split}
			\end{equation}
			where throughout the computation we use the Bianchi identities and Proposition~\ref{prop:torsion_twisted_curvature} expressing $\Omega^T$ via $\Omega$ and $\n T$.
			This proves the lemma.
		\end{proof}
		Let 
		\[
		R_{\mc E}\in \Lambda^2(M)\otimes \End(\mc E)
		\]
		be the curvature of $\n_{\mc E}$. Lemma~\ref{lm:torsion_twisted_connection_evolution} directly implies an analogous formula for the evolution of $\n_{\mc E}(t)$. 
		\begin{equation}\label{eq:nT_evolution}
		\begin{split}
		\frac{d}{dt}(\n_{\mc E})_i&=
		-g^{m\bar n}(\widetilde{\n}_m R_{\mc E})_{i\bar n}
		-\frac{1}{2}g^{m\bar n}g^{p\bar s}T_{\bar n\bar s i}(R_{\mc E})_{mp}\\
		&=
		g^{m\bar n}(\widetilde{\n}_m R_{\mc E})_{\bar n i}
		-\frac{1}{2}g^{m\bar n}g^{p\bar s}T_{\bar n\bar s i}(R_{\mc E})_{mp},
		\end{split}
		\end{equation}
		where $\widetilde{\n}$ acts as $\n\otimes 1+1\otimes\n_{\mc E}$ on $\Lambda^2(M)\otimes \End(\mc E)$.
		
		To write down the evolution equation for $\pmb\beta$, let us fix a $C^\infty$-splitting of short exact sequence~\eqref{eq:short_exact_sequence}:
		\[
		\begin{tikzcd}
			0\arrow{r} & \mc V\arrow{r}{\iota} & \mc{E}\arrow{r}{p}\arrow[bend left=33,dashrightarrow]{l}{\iota^*} & \mc Q\arrow{r}\arrow[bend left=33,dashrightarrow]{l}{p^*} & 0,
		\end{tikzcd}
		\]
		and consider an isomorphism
		\begin{equation}\label{eq:isomorphism_split}
		\iota\oplus p^*\colon \mc V\oplus\mc Q \xrightarrow{\simeq} \mc E
		\end{equation}
		With this identification, connection $\n_{\mc E}$ corresponds to:
		\begin{equation}\label{eq:isomorphism_connection}
		(\iota^*\oplus p)\circ\n_{\mc E}\circ(\iota\oplus p^*)=
		\left(
		\begin{matrix}
		\n_{\mc V} & \pmb\gamma \\
		\pmb\beta & \n_{\mc Q}
		\end{matrix}
		\right),
		\end{equation}
		where
		\begin{equation}
		\begin{split}
		\n_{\mc V}&=\iota^*\circ\n_{\mc E}\circ\iota\\
		\n_{\mc Q}&=p\circ\n_{\mc E}\circ p^*
		\end{split}
		\end{equation}
		are the connections induced by $\n_{\mc E}$ on $\mc V$ and $\mc Q$ respectively via isomorphism~\eqref{eq:isomorphism_split},
		and
		\begin{equation}
		\begin{split}
		\pmb\beta&=p\circ\n_{\mc E}\circ \iota,
			\quad\pmb\beta\in \Lambda^1(M)\otimes\Hom(\mc V,\mc Q)\\
		\pmb\gamma&=\iota^*\circ\n_{\mc E}\circ p^*,
			\quad\pmb\gamma\in \Lambda^1(M)\otimes\Hom(\mc Q,\mc V)
		\end{split}
		\end{equation}
		are the second fundamental form and its conjugate. Since $\n_{\mc E}$ is a holomorphic connection, and $\mc V\subset\mc E$ is a holomorphic subbundle, we have $\pmb\beta\in\Lambda^{1,0}(M)\otimes\Hom(\mc V,\mc Q)$.
		
		Identity~\eqref{eq:isomorphism_connection} implies the corresponding decomposition for $R_{\mc E}=(\n_{\mc E})^2$ under isomorphism~\eqref{eq:isomorphism_split}:
		\begin{equation}\label{eq:isomorphism_curvature}
		(\iota^*\oplus p)\circ R_{\mc E}\circ(\iota\oplus p^*)=
		\left(
		\begin{matrix}
		R_{\mc V}+\pmb\gamma\wedge\pmb\beta & \n_{\Hom(\mc Q,\mc V)}\pmb\gamma\\
		\n_{\Hom(\mc V,\mc Q)}\pmb\beta & R_{\mc Q}+\pmb\beta\wedge\pmb\gamma
		\end{matrix}
		\right),
		\end{equation}
		with $R_{\mc V}$ and $R_{\mc Q}$~--- the curvatures of $(\mc V,\n_{\mc V})$ and $(\mc Q,\n_{\mc Q})$ respectively, and 
		\begin{equation}
		\begin{split}
		\n_{\Hom(\mc V,\mc Q)}\colon \Lambda^{1}(M)\otimes\Hom(\mc V,\mc Q)\to \Lambda^{2}(M)\otimes\Hom(\mc V,\mc Q)\\
		\n_{\Hom(\mc Q,\mc V)}\colon \Lambda^{1}(M)\otimes\Hom(\mc Q,\mc V)\to \Lambda^{2}(M)\otimes\Hom(\mc Q,\mc V)
		\end{split}
		\end{equation}
		are natural extensions of the corresponding connections. Note that, since connection $\n_{\Hom(\mc V,\mc Q)}$ is holomorphic, and $\pmb\beta\in\Lambda^{1,0}(M)\otimes \Hom(\mc V,\mc Q)$, we have $(\n_{\Hom(\mc V,\mc Q)})^{0,1}\pmb\beta=\bar{\pd}\pmb\beta$
		
		Evolution equation~\eqref{eq:nT_evolution} implies that
		\[
		\frac{d}{dt}\pmb \beta_i=
		p\circ
		\Bigl(g^{m\bar n}(\widetilde{\n}_m R_{\mc E})_{\bar n i}
		-\frac{1}{2}g^{m\bar n}g^{p\bar s}T_{\bar n\bar s i}(R_{\mc E})_{mp}\Bigr)
		\circ\iota.
		\]
		Now we apply equations~\eqref{eq:isomorphism_connection} and~\eqref{eq:isomorphism_curvature}, to identify the piece of $g^{m\bar n}(\widetilde{\n}_m R_{\mc E})_{\bar n i}
		-\frac{1}{2}g^{m\bar n}g^{p\bar s}T_{\bar n\bar s i}(R_{\mc E})_{mp}\in\Lambda^{1,0}(M)\otimes\End(\mc E)$ which corresponds to $\Lambda^{1,0}(M)\otimes\Hom(\mc V,\mc Q)$. It is straightforward to check that
		\begin{equation}\label{eq:beta_evolution}
		\begin{split}
		\frac{d}{dt}\pmb \beta_i=({\bar\pd}^* \bar{\pd}\pmb\beta)_i+(A*\pmb\beta)_i+(B*\n_{\Hom(\mc V,\mc Q)}\pmb\beta)_i
		\end{split}
		\end{equation}
		for some contractions with time-dependent tensors $A,B$.
		
		Equation~\eqref{eq:beta_evolution} is a strictly parabolic PDE for $\pmb\beta\in \Lambda^{1,0}(M)\otimes\Hom(\mc V,\mc Q)$ with time-dependent coefficients. This equation has a unique forward solution $\pmb\beta(t)=0$ on $[\tau',\tau)$ starting with $\pmb\beta(\tau')=0$ by the standard theory of parabolic PDEs. It also has a unique backward solution on $[0,\tau']$, by a general result~\cite[Theorem 3]{ko-16}. Therefore $\pmb\beta(t)=0$ and $\mc V\subset \mc E$ is invariant under $\n_{\mc E}(t)$ for all $t\in[0,\tau)$.
	\end{proof}
	
	\begin{theorem}\label{thm:nt_invariance_2}
		With the notations of Theorem~\ref{thm:nt_invariance}, if a holomorphic subbundle $\mc V\subset \mc E$ is trivialized by $\n_{\mc E}(\tau')$-parallel sections for some $\tau'\in[0,\tau)$, then the same is true for $\n_{\mc E}(t)$ for any $t\in[0,\tau)$.
	\end{theorem}
	\begin{proof}
		The proof is essentially the same as for Theorem~\ref{thm:nt_invariance_2}. By the above theorem, we already know that $\mc V$ is invariant under $\n_{\mc E}(t)$ for any $t\in[0,\tau)$. Now, for a section $s\in\Gamma(\mc V)$, we use equation~\eqref{eq:nT_evolution}, and directly verify that $\n_{\mc E}(t)s\in \Gamma(\mc V)$ satisfies a linear strictly parabolic PDE, so $\n_{\mc E}(t)s\equiv 0$ is the unique backward/forward solution starting at $0$.
	\end{proof}
	
	\begin{remark}
		Theorem~\ref{thm:nt_invariance_2} for $\mc V=K_M$ was first observed in~\cite{us-17-2}. There we proved that the minimum of the second scalar Chern curvature
		\[
		\widehat s:= \Omega_{n\bar s}^{\ \ \bar s n}
		\]
		is strictly increasing along the HCF $g(t), t\in[0,\tau)$, unless $\widehat{s}(t)\equiv 0$, and  the canonical bundle $K_{\til M}$ admits a $\n^T(t)$-parallel section for any $t\in[0,\tau)$.
	\end{remark}

	\section{Geometry of Semipositive Griffiths Curvature}\label{sec:main}
	
	In this section we apply the above observations on the geometry of the torsion-twisted connection, and prove our main structure result about Hermitian manifolds of semipositive Griffiths curvature. Before we turn to the proof, let us consider in detail the principle source of such manifolds.
	
	\begin{example}[Submersion metrics on complex homogeneous manifolds]\label{ex:submersion}
		Let $G$ be a connected complex Lie group, $H\subset G$ its complex subgroup, and $\mf h\subset \mf g$ the corresponding Lie algebras. Denote the connected component of the normalizer of $H$ in $G$ by $N^0_G(H)$.
		
		Consider a homogeneous manifold $M=G/H$. The short exact sequence of complex vector spaces
		\[
		0\to \mf h\to\mf g\to \mf g/\mf h\to 0 
		\]
		defines a short exact sequence of holomorphic vector bundles on $M$. Specifically, at $[\gamma H]\in M$ we have
		\[
		0\to \Ad_{\gamma}\mf h\xrightarrow{i} \mf g\xrightarrow{p} T^{1,0}_{[\gamma H]}M\to 0,
		\]
		where $\mf g\to M$ is a trivial bundle, and $\Ad(\mf h):=\{\Ad_{\gamma}\mf h\}_{[\gamma H]\in M}$ is its subbundle. 
		Any Hermitian metric $h\in\Sym^{1,1}(\mf g^*)$ on the Lie algebra of $G$ defines a \emph{submersion metric} on $M$ via $p\colon \mf g\to T^{1,0}M$, which we denote by $g=p_*h$.
		
		As we have computed in~\cite{us-th}, $(M,p_*h)$ has semipositive Griffiths curvature, and its associated Chern-Ricci form $\rho$ can be expressed as follows. Let $\xi\in T^{1,0}_{m}(G/H)$ and for simplicity assume $m=[eH]\in M$ is the coset the unit.
		Then
		\[
		\rho(\xi,\bar{\xi})=\big|\beta(\xi)\big|^2_{\Hom(\mf h,\mf g/\mf h)},
		\]
		where $\beta\in \Lambda^{1,0}(M,\Hom(\Ad(\mf h),T^{1,0}M))\simeq \mf g/\mf h\otimes \Hom(\mf h,\mf g/\mf h))$ is given by
		\[
		\beta_\xi(w)=p([v,w]),\quad \mbox{where }\xi=p(v)\in T_{[eH]}^{1,0}M,
		\]
		and the norm is taken with respect to the metric induced by $h$.
		
		Clearly, $\rho(\xi,\bar \xi)=0$ if and only if $[v,\mf h]\subset \mf h$, or, equivalently, $\exp(v)\in G$ normalizes $\mf h$. Therefore, the null spaces of $\rho$ are generated by the infinitesimal action of the normalizer of $H$ in $G$. It is easy to see that the action of $N^0_G(H)/H$ on $M=G/H$ is almost free, holomorphic, and $p_*h$-isometric.		
	\end{example}
	
	Example~\ref{ex:submersion} highlights what we can and what we cannot anticipate from a possible characterization of Hermitian manifolds of semipositive Griffiths curvature. There is no hope for the isometric splitting of such a manifold $(M,g)$, akin the one in K\"ahler setting, since the orbits of $N^0_G(H)/H$ in the example do not split off isometrically/holomorphically. Instead, we expect (at least locally) a holomorphic bundle/foliation-like structure with the fibers/leaves~--- orbits of a holomorphic isometric action of a Lie group. The next theorem formalizes this expected behavior.
	
	\begin{theorem}\label{thm:main}
		Let $(M,g_0)$ be a compact Hermitian manifold with semipositive Griffiths curvature. Then there exists a complex Lie group $G$ acting almost freely, holomorphically and $g_0$-isometrically on the universal cover $\widetilde{M}$ of $M$ with the following property:
		\smallskip
		
		\noindent
		For any $\epsilon>0$, $k\in \N$ there exists a $G$-invariant Hermitian metric $g$ on $M$ such that
		\begin{enumerate}
			\item[(i)] $||g_0-g||_{C^k,g_0}< \epsilon$;
			
			\item[(ii)] $(M, g)$ has semipositive Griffiths curvature;
			
			\item[(iii)] the null space distribution of the Chern-Ricci form $\rho=\rho(g)$ coincides with the 		
			distribution generated by the infinitesimal action of $G$:
			\[\mathrm{Null}(\rho_x)=T_x (G.x),\]
			where $G.x$ is the orbit of $x\in\widetilde{M}$.
		\end{enumerate}
	\end{theorem}
	\begin{proof}		
		While the metric $g_0$ satisfies a strong semipositivity notion, we do not have any a priori control over the null spaces of its Chern-Ricci form $\rho(g_0)$. In particular, these null spaces could potentially have different rank at different points, and even if the rank is constant, there is no reason for the corresponding distribution in $T^{1,0}M$ to be holomorphic/integrable. To overcome this issue we \emph{regularize} metric $g_0$ with the HCF. 		
		The idea of the proof is to solve the HCF on $(M,g_0)$, and show that the evolved metric $g(t)$ satisfies the statement of the theorem for sufficiently small time slice $t>0$.
		
		Let $g(t)$ be the solution to the HCF om $(M,g)$ for $t\in[0,\tau)$. Property (i) is clearly satisfied for small $t$, since the HCF is a strictly parabolic equation for $g(t)$ starting with $g(0)=g_0$. Property (ii) is satisfied by the main result of~\cite{us-16}, which states that the HCF preserves Griffiths semipositivity of the Chern curvature, see Theorem~\ref{thm:hcf-griffiths-preserved}.
		
		The most interesting and nontrivial is Property (iii), as it guarantees the existence of holomorphic Killing vector fields on $\til M$. To prove it, we consider metric $g:=g(t)$ for any $t>0$. Our plan is to study the geometry of the torsion-twisted connection $\n^T$ on the universal cover $(\til M,g)$. By the strong maximum principle for HCF~\cite[Theorem 5.2]{us-16} (see also~\cite[Theorem 5.1]{us-17-2})
		\begin{itemize}
			\item the set
			\[
			Z_x=\{\xi\otimes\bar{\eta}^\#\,|\,\xi,\eta\in T_x^{1,0}\til M,\ \ \Omega(\xi,\bar{\xi},\eta,\bar{\eta})=0\}\subset \End(T_x^{1,0}\til M),
			\]
			where
			\[
			\bar{\eta}^\#:=g(\cdot,\bar{\eta})\in\Lambda^{1,0}(\til M),
			\]
			is invariant under $\n^T$;
			\item for  any $\xi\otimes\bar{\eta}^\#\in Z_x$
			\[
			\Omega(\xi,\cdot,\bar{\cdot},\bar{\eta})=0,\quad  g(\n_\xi  T(\cdot,\cdot),\bar{\eta})=0.
			\]
		\end{itemize}
		
		In the view of Proposition~\ref{prop:torsion_twisted_curvature}, we can interpret the second statement as
		\[
			g(\Omega^T(\cdot,\cdot)\xi,\bar{\eta})=0
		\]
		Now, since $\Omega$ is Griffiths semipositive, the Chern-Ricci form
		\[
		\rho_x(\xi,\bar{\xi})=\sum_{i=1}^n \Omega(\xi,\bar{\xi},e_i,\bar{e_i}),\quad \mbox{where } \{e_1,\dots,e_n\} \mbox{ is a unitary basis of } T_x^{1,0}\til M,
		\]
		is also semipositive, and the set
		\[
		\mathrm{Null}({\rho}_x)=\{\xi\in T_x^{1,0}\ |\ \rho_x(\xi,\bar{\xi})=0 \}=\{\xi\in T_x^{1,0}\ |\ \Omega(\xi,\bar{\xi},\eta,\bar{\eta})=0 \mbox{ for any } \eta\in T_x^{1,0}\til M\}
		\]
		is invariant under $\n^T$-parallel transport. Moreover, for any $\xi\in \mathrm{Null}(\rho_x)$ we have
		\[
		\Omega^T(\cdot,\cdot)\xi=0\quad (\mbox{or equivalently }\Omega(\xi,\cdot,\cdot,\cdot)=0,\ \n_\xi T(\cdot,\cdot)=0),
		\]
		and, vice versa, any $\xi\in T_x^{1,0}\til M$ such that $\Omega^T(\cdot,\cdot)\xi=0$ belongs to $\mathrm{Null}(\rho_x)$. Therefore by Ambrose-Singer theorem
		\[
		\mathrm{Null}(\rho_x):=\{\xi\in T_x^{1,0}\til M\ |\ \mathrm{Hol}^0_x(\n^T)\xi=\xi\}.
		\]
		Since $\til M$ is simply connected, $\mathrm{Hol}^0_x(\n^T)=\mathrm{Hol}_x(\n^T)$. 
		
		Now, we are in a good shape to apply Theorem~\ref{thm:G_action} and conclude that there exists an almost free, holomorphic, $g$-isometric action of a complex Lie group $G$ on $\til M$, such that
		\[
		\mathrm{Null}(\rho_x)=T_x(G.x).
		\]
		
		It remains to check that the action of $G$ is also $g_0$-isometric. By Theorem~\ref{thm:nt_invariance_2} the distribution $\mc F:=\mathrm{Null}(\rho_x)$ has the same $\n^T(t)$-parallel trivialization for any metric $g(t)$, $t\in[0,\tau)$ along the HCF, including $g(0)=g_0$. Therefore group $G$ is independent of the choice $t\in[0,\tau)$, and the action of $G$ is $g(t)$-isometric for any $t\in[0,\tau)$, including $g_0$.
	\end{proof}
	
	\begin{remark}
		The foliation of $\til M$ by $G$-orbits admits a \emph{transverse-K\"ahler} form. Indeed the Chern-Ricci form $\rho$ of $(\til{M}, g)$ is closed and strictly positive on the complement of distribution $\mc F=\mathrm{Null}(\rho)$. Similar foliations and the corresponding transverse-K\"ahler structures have been proved to be very effective tools in non-K\"ahler geometry, see e.g.,~\cite{or-ve-11,pa-us-ve-16}.
	\end{remark}
	
	In general, group $G$ in~Theorem~\ref{thm:main} can be trivial. This is the case, e.g., for any rational homogeneous manifold equipped with a submersion metric, since in Example~\ref{ex:submersion} we would have $N_G(H)=H$. However, if this is the case, we deduce that the Chern-Ricci form $\rho$ of $(M,g)$ is strictly positive, in particular the anticanonical bundle $-K_M$ is ample, and $M$ is projective (even Fano).
	
	\begin{corollary}
		If a compact complex \emph{non-Fano} manifold $M$ admits a metric $g_0$ of semipositive Griffiths curvature, then there exists an almost free, holomorphic, and isometric action of a \emph{nontrivial} complex Lie group $G$  on the universal cover of $M$.
	\end{corollary}
	\emergencystretch=1em
	\printbibliography
\end{document}